\documentclass[11pt]{amsart}

\usepackage{amsmath}
\usepackage{amsxtra}
\usepackage{amscd}
\usepackage{amsthm}
\usepackage{amsfonts}
\usepackage{amssymb}
\usepackage{mathtools}
\usepackage{eucal}
\usepackage[all]{xy}
\usepackage{graphicx}
\usepackage{tikz}
\usetikzlibrary{shapes.misc, matrix, arrows, decorations.pathmorphing, fit, backgrounds, decorations.markings, arrows.meta}
\usepackage{tikz-cd}
\usepackage{mathrsfs}
\usepackage{euler}
\usepackage{color}
\usepackage{caption}
\usepackage{enumerate}
\usepackage{faktor}
\usepackage{bbm}
\usepackage{stmaryrd}
\usepackage[T1]{fontenc} % improved font encoding
\usepackage[utf8]{inputenc} % for better handling of non-ASCII characters
\usepackage{hyperref}
\usepackage{float}
\usepackage{multicol}
\usepackage{scrextend}
\usepackage{dynkin-diagrams}

\usepackage[bottom=1.25in, top=1.5in, right=1.5in, left=1.5in]{geometry}

\theoremstyle{plain}
  \newtheorem{theorem}{Theorem}
  \newtheorem{proposition}[theorem]{Proposition}
  \newtheorem{lemma}[theorem]{Lemma}

\theoremstyle{definition}
  \newtheorem{definition}[theorem]{Definition}
  \newtheorem{example}[theorem]{Example}
  \newtheorem{conjecture}[theorem]{Conjecture}

\theoremstyle{remark}
  \newtheorem{remark}[theorem]{Remark}
  
  \newtheorem*{remark*}{Remark}

\DeclareMathAlphabet{\mathcal}{OMS}{cmsy}{m}{n}

\newcommand{\A}{\mathbb{A}}

\newcommand{\N}{\mathbb{N}}

\newcommand{\Z}{\mathbb{Z}}

\newcommand{\SL}{{\mathrm{SL}}}

\newcommand{\frieze}{{\mathrm{Frieze}}}

\newcommand{\paren}[1]{\mathopen{}\left(#1\right)\mathclose{}}
\newcommand{\set}[1]{\mathopen{}\left\{#1\right\}\mathclose{}}

\newcommand{\verts}[1]{\mathopen{}\left\lvert#1\right\rvert\mathclose{}}

\newcommand\restr[2]{{% we make the whole thing an ordinary symbol
  \left.\kern-\nulldelimiterspace % automatically resize the bar with \right
  #1 % the function
  \right|_{#2} % this is the delimiter
  }}
\newcommand{\Mid}{\,\middle|\,}

\newcommand{\widesim}[2][2]{
  \mathrel{\overset{#2}{\scalebox{#1}[1]{$\sim$}}}
}

\makeatletter
\AtBeginDocument
 {
    \def\@thm#1#2#3{%
      \ifhmode
        \unskip\unskip\par
      \fi
      \normalfont
      \trivlist
      \let\thmheadnl\relax
      \let\thm@swap\@gobble
      \let\thm@indent\indent % indent
      \thm@headfont{\scshape}% heading font small caps
      \thm@notefont{\fontseries\mddefault\upshape}%
      \thm@headpunct{.}% add period after heading
      \thm@headsep 5\p@ plus\p@ minus\p@\relax
      \thm@space@setup
      #1% style overrides
      \@topsep \thm@preskip               % used by thm head
      \@topsepadd \thm@postskip           % used by \@endparenv
      \def\dth@counter{#2}%
      \ifx\@empty\dth@counter
        \def\@tempa{%
          \@oparg{\@begintheorem{#3}{}}[]%
        }%
      \else
        \H@refstepcounter{#2}%
        \hyper@makecurrent{#2}%
        \let\Hy@dth@currentHref\@currentHref
        \AddToHookNext{para/begin}{\MakeLinkTarget*{\Hy@dth@currentHref}}%
        \def\@tempa{%
          \@oparg{\@begintheorem{#3}{\csname the#2\endcsname}}[]%
        }%
      \fi
      \@tempa
    }%
  \dth@everypar={%
    \@minipagefalse \global\@newlistfalse
    \@noparitemfalse
    \if@inlabel
      \global\@inlabelfalse
      \begingroup \setbox\z@\lastbox
       \ifvoid\z@ \kern-\itemindent \fi
      \endgroup
      \unhbox\@labels
    \fi
    \if@nobreak \@nobreakfalse \clubpenalty\@M
    \else \clubpenalty\@clubpenalty \everypar{}%
    \fi
  }%
}

\newcommand{\subalign}[1]{%
  \vcenter{%
    \Let@ \restore@math@cr \default@tag
    \baselineskip\fontdimen10 \scriptfont\tw@
    \advance\baselineskip\fontdimen12 \scriptfont\tw@
    \lineskip\thr@@\fontdimen8 \scriptfont\thr@@
    \lineskiplimit\lineskip
    \ialign{\hfil$\m@th\scriptstyle##$&$\m@th\scriptstyle{}##$\hfil\crcr
      #1\crcr
    }%
  }%
}

\makeatother

\tikzset{
  /Dynkin diagram/root radius=.12cm,
  /Dynkin diagram/edge/.style={thick}
}

\title{On maximal Dynkin friezes}

\author[Robin Zhang]{Robin Zhang}
\address{Department of Mathematics, Massachusetts Institute of Technology}
\email{robinzmath@gmail.com}

\date{September 16, 2026}

\begin{document}

% \linenumbers

\begin{abstract}
  The maximal entries of Dynkin friezes over the positive integers have recently been determined for all finite Dynkin types except $B_n$ and $D_n$. In this note, we explicitly construct large positive integral points on affine cluster varieties of type $B_n$ (resp. $D_n$), giving rise to friezes of types $B_n$ (resp. $D_n$) over the positive integers with largest entries $F_{n+1} F_{n+2} - 1$ (resp. $F_n F_{n+1} - 1$) where $F_k$ is the $k$-th Fibonacci number. We conjecture that these are the maximal possible entries for their respective Dynkin types.
\end{abstract}

\maketitle

\setcounter{tocdepth}{1}
\tableofcontents

%%%%%%%%%%%%%%%%%%%%%%%%%%%% Introduction %%%%%%%%%%%%%%%%%%%%%%%%%%%%

% \vspace*{-.2in}

\section{Introduction}
In combinatorics,
frieze patterns were introduced
by Coxeter \cite{coxeter} to generalize
the classical identities of Gauss \cite{gauss-1866}
on the \textit{pentagramma mirificum}.
The classical frieze patterns of Coxeter
were further generalized by Caldero--Chapoton \cite{caldero-chapoton}
(cf. Assem--Reutenauer--Smith \cite{assem-reutenauer-smith})
to arbitrary Dynkin types:
a (positive integral) Dynkin frieze is a homomorphism of
an acyclic cluster algebra with trivial coefficients
to $\Z$ that send cluster variables to positive integers,
and Coxeter's frieze patterns are precisely the
Dynkin friezes of type $A_n$.
Any Dynkin frieze can be represented
as an infinite array of Dynkin diagrams
with positive integers at each node
and a unimodular condition.
For example,
Coxeter's classical friezes of type $A_n$
are arrays
\begin{equation}
  \label{eq:An-frieze}
  \begin{tikzcd}[row sep=.01in, column sep = .03in, font={\small}]
      \ldots & & 1 & & 1 & & 1 & & 1 & & 1 & & \ldots \\
      & \ldots & & x_{1, 1} & & x_{1, 2} & & x_{1, 3} & & x_{1, 4} & & \ldots \\
      \ldots & & x_{2, 0} & & x_{2, 1} & & x_{2, 2} & & x_{2, 3} & & x_{2, 4} & & \ldots \\
      & \ddots & & \ddots & & \ddots & & \ddots & & \ddots & & \ddots \\
      \ldots & & x_{n, -1} & & x_{n, 0} & & x_{n, 1} & & x_{n, 2} & & x_{n, 3} & & \ldots \\
      & \ldots & & 1 & & 1 & & 1 & & 1 & & \ldots 
  \end{tikzcd}
\end{equation}
in which every adjacent sub-diamond
$\begin{smallmatrix}
  & b & \\
  a & & d \\
  & c &
\end{smallmatrix}$
satisfies the
$\SL_2$ unimodular relation $ad - bc = 1$.
For arbitrary Dynkin types, these unimodular rules
are modified according to the underlying generalized Cartan matrix;
we give the precise combinatorial definitions, along with the
explicit formal arrays for types $B_n$ and $D_n$,
in Section \ref{sec:frieze-review}.
Alternatively, we may view positive integral friezes
as positive integral points on algebraic varieties;
we review this perspective in Section \ref{sec:frieze-review} as well.

For a fixed Dynkin type $\Delta_n$ (e.g. $\Delta_n = A_3$),
let $\frieze(\Delta_n, \N)$ denote the set of
positive integral Dynkin friezes of type $\Delta_n$.
For the classical family $A$,
the finiteness and enumeration of
$\frieze(A_n, \N)$
was established by
Conway--Coxeter \cite{coxeter-conway-2}
via an explicit
bijection to the set of triangulations
of $(n+3)$-gons,
which has cardinality equal to
the $(n+1)$-st Catalan number
$\frac{1}{n+2}
{\binom{2n+2}{n+1}}$.
For any infinite type $\Delta_n$,
the infinitude of $\frieze(\Delta_n, \N)$
was established by Morier--Genoud
\cite{morier-genoud-2012}
using the classification of cluster algebras
of Fomin--Zelevinsky
\cite{fomin-zelevinsky-2}.
For a general finite type $\Delta_n$,
the finiteness of
$\frieze(\Delta_n, \N)$
was first proved by Gunawan--Muller
\cite{gunawan-muller};
the complete enumeration of
$\#\frieze(\Delta_n, \N)$ for all $\Delta_n$
is the result of an extensive program
in combinatorics, cluster algebra theory,
and Diophantine geometry
in \cite{coxeter-conway-2,morier-genoud-ovsienko-tabachnikov,
fontaine-plamondon,bfgst,zhang-2025}
and summarized in Figure \ref{fig:enumeration}.
\begin{figure}[H]
  % (Sharp except $B_n \, \& \, D_n$)\\
  \renewcommand{\arraystretch}{1.4}
  \resizebox{\textwidth}{!}{%
    \begin{tabular}{|c|c|c|c|c|c|c|c|c|c|c|c|}
      \hline
      Type $\Delta_n$
          & $A_n$ & $B_n$ & $C_n$ & $D_n$ 
          & $E_6$ & $E_7$ & $E_8$ 
          & $F_4$ & $G_2$ 
          & \textcolor{red}{$\infty$} \\
      \hline
      $\#\frieze(\Delta_n, \N)$
          & $\displaystyle \frac{1}{n+2}{\binom{2n+2}{n+1}}$
          & $\displaystyle \sum_{m=1}^{\lfloor \sqrt{n+1} \rfloor}
        \binom{2n - m^2 + 1}{n}$
          & $\displaystyle \binom{2n}{n}$
          & $\displaystyle \sum_{m=1}^{n} d(m) \binom{2n - m - 1}{n - m}$
          & $868$ 
          & $4400$ 
          & $26952$ 
          & $112$ 
          & $9$ 
          & \textcolor{red}{$\infty$} \\
      \hline
    \end{tabular}
  }
  \caption{Enumeration of Dynkin friezes,
    $d(m) := \#\{\text{divisors of } m\}$}
  \label{fig:enumeration}
\end{figure}

In this paper, we study the set of values attained by Dynkin friezes.
Our first result is rather simple and
concerns the ``universality'' of frieze values.
While individual friezes are constrained by local relations,
we show that the union of all friezes covers the natural numbers in Section \ref{sec:universality}.
\begin{theorem}
  \label{thm:universality}
  For every integer $m \geq 2$
  and every integer $n \geq m$,
  there exists a positive integral frieze
  of type $A_n$ (resp. $B_n, C_n, D_n$)
  containing $m$ as an entry.
  Consequently, the set of all values appearing in
  positive integral Dynkin friezes is exactly $\N$.
\end{theorem}

Given that the value set is unbounded as the rank $n$ varies,
a natural question arises:
what is the maximal entry in a frieze of \textit{fixed} rank $n$?
The maxima for types $A_n$ and $C_n$
are known to be
$F_{n+2}$ and $F_{2n+1}$
respectively due to the work of
Cheah--de Saint Germain
\cite[Theorem 1]{cheah-dsg},
where $(F_k)_{k \geq 0}$ is
the Fibonacci sequence defined by
$F_0 = 0$, $F_1 = 1$,
and $F_{k + 1} = F_k + F_{k - 1}$.
The maxima for exceptional types are known
due to the author's explicit calculations
in \cite[Proposition 4.1]{zhang-2025}.
For types $B_n$ and $D_n$,
the problem remains open.
\begin{figure}[H]
  \renewcommand{\arraystretch}{1.4}
  \resizebox{\textwidth}{!}{%
    \begin{tabular}{|c|c|c|c|c|c|c|c|c|c|c|c|}
      \hline
      Type $\Delta_n$ 
          & $A_n$ & $B_n$ & $C_n$ & $D_n$ 
          & $E_6$ & $E_7$ & $E_8$ 
          & $F_4$ & $G_2$ 
          & \textcolor{red}{$\infty$} \\
      \hline
      $u_{\Delta_n, \N}$
          & $F_{n+2}$ 
          & $\leq 2^{\frac{(n+1)^2(n-2)}{2}}$
          & $F_{2n+1}$ 
          & $\leq 2^{\frac{n^3}{2}}$
          & $307$ 
          & $135503$ 
          & $2820839$ 
          & $307$ 
          & $14$ 
          & \textcolor{red}{$\infty$} \\
      \hline
    \end{tabular}
  }
  \caption{Explicit bounds on maximal frieze entries from
    \cite{cheah-dsg} and \cite[Proposition 4.1]{zhang-2025}}
  \label{fig:bounds}
\end{figure}

Our second main contribution is the construction of explicit families of friezes
for types $B_n$ and $D_n$ that exhibit rapid growth.
For a fixed Dynkin type $\Delta_n$, define
\[
  u_{\Delta_n, \N}
    := \max \set{x_{i, j} \in F
      \Mid F \in \frieze(\Delta_n, \N)}
\]
to be the largest positive integer that can
appear as an entry of
a positive integral frieze of type $\Delta_n$.
Our new explicit constructions
of $B_n$-friezes and $D_n$-friezes
in Section \ref{sec:maximal}
give lower bounds
on $u_{B_n, \N}$ and $u_{D_n, \N}$,
which complement the exponential upper bounds
of \cite[Proposition 4.1]{zhang-2025}.

\begin{theorem}
  \label{thm:max-bound}
  Let $n$ be a positive integer.
  Then
  \begin{align*}
    u_{B_n, \N}
      &\geq F_{n+1} F_{n+2} - 1, \\
    u_{D_n, \N} &\geq F_n F_{n+1} - 1.
  \end{align*}
\end{theorem}

Based on computational verification
for small ranks
(see Appendix \ref{app:maximal}),
we conjecture that our constructions
are actually optimal.
\begin{conjecture}
  \label{conj:BnDn-largest}
  Let $n$ be a positive integer.
  Then
  \begin{align*}
    u_{B_n, \N} &= F_{n+1} F_{n+2} - 1, \\
    u_{D_n, \N} &= F_n F_{n+1} - 1.
  \end{align*}
\end{conjecture}

\begin{remark}
  \label{rem:efficiency}
  Theorem \ref{thm:universality} guarantees that for each of the four classical Dynkin families
  $\Delta \in \{A, B, C, D\}$, any integer $k \geq 2$ appears as an entry in a Dynkin frieze of rank $k$.
  This raises natural quantitative questions regarding the ``efficiency'' of
  each classical Dynkin family; for each positive integer $k$,
  consider the minimal rank
  \[
    n_{\min}(\Delta, k)
      = \min
        \set{n \geq 2 \Mid
        k \in F
        \text{ for some frieze }
        F \in \frieze(\Delta_n, \N)}
  \]
  in which $k$ appears as an entry.
  An immediate corollary of Theorem \ref{thm:universality}
  is the upper bound
  $n_{\min}(\Delta, k) \leq k$
  for each classical family $\Delta$
  and all integers $k \geq 2$.
  As we elaborate in Section \ref{sec:exponential-growth},
  we can only deduce logarithmic lower bounds
  on $n_{\min}(\Delta, k)$
  for each classical family $\Delta$
  and all integers $k \geq 2$
  due to the rapid growth
  of maximal Dynkin friezes.
  We leave it as an open question
  to determine
  the asymptotic behavior of
  $n_{\min}(\Delta, k)$
  for each classical family $\Delta$.
\end{remark}

\subsection*{Acknowledgements}

The author is grateful to Kannan Soundararajan
for posing a question at
the Stanford Number Theory Seminar
that motivated this study of
maximal frieze entries of types $B_n$ and $D_n$.
Thanks are also due to Yeuk Hay Joshua Lam and Gregorio Baldi for a discussion
that helped identify an error
in Example~\ref{ex:D5-frieze},
to Jon Cheah for a suggestion that led to
a more complete Figure~\ref{table:maximal-frieze-data},
and to the anonymous referees for comments
that helped refine the paper's exposition.

The author is supported
by the National Science Foundation
under Grant No. DMS-2303280.

%%%%%%%%%%%%%%%%%%%%%%%%%%%%%%%%%%%%%%%%%%%%%%%%%%%%%%%%%%%%%%%%%%

\numberwithin{equation}{section}
\numberwithin{theorem}{section}
\numberwithin{figure}{section}

%%%%%%%%%%%%%%%%%%%%%%%%%%%%%%%%%%%%%%%%%%%%%%%%%%%%%%%%%%%%%%%%%%

\section{Frieze patterns and frieze points}
\label{sec:frieze-review}

\subsection{Frieze patterns}
\label{sec:frieze-patterns}

While Dynkin friezes originate in the theory of cluster algebras, they can be defined purely combinatorially via recurrence relations. 

\begin{definition}[Dynkin frieze]
  \label{def:dynkin-frieze}
  Fix a finite Dynkin type $\Delta_n$
  with generalized Cartan matrix
  $C = (c_{i,j})_{1 \leq i, j \leq n}$.
  A \textit{positive integral frieze
  of type $\Delta_n$}
  is an array of positive integers
  $(x_{i,j})_{1 \leq i \leq n, j \in \Z}$
  satisfying the unimodular relation
  \[
    x_{i, j} x_{i, j + 1}
      - \prod_{k < i} x_{k, j + 1}^{\verts{c_{k,i}}}
        \prod_{k > i}x_{k,j}^{\verts{c_{k,i}}} = 1
  \]
  for all $i \in \{1, \ldots, n\}$ and $j \in \Z$.
\end{definition}
Typically, a Dynkin frieze is visually represented
as an array with $n$ rows and
(countably) infinitely many diagonal columns
each in the shape of the given Dynkin diagram,
with positive integer values at each node
of each copy of the Dynkin diagram.
The array consists of $n$ rows
and infinitely many columns laid out diagonally,
with additional virtual top and bottom rows
consisting entirely of ones.
The general unimodular relation
in Definition \ref{def:dynkin-frieze}
dictates the local sub-diamond conditions
for different Dynkin types,
e.g. the $\SL_2$ condition for classical
$A_n$ friezes in Array \ref{eq:An-frieze}.
We now specialize this to the two classical families of interest in this paper.

For type $B_n$ ($n \geq 2$),
the Cartan matrix dictates
that the usual $\SL_2$ rule is
modified at the double edge of the Dynkin diagram.
With the Cartan matrix and
corresponding labeling
for the Dynkin diagram of type $B_n$
\[
  \setlength{\arraycolsep}{4pt} 
  \renewcommand{\arraystretch}{0.5}
  C = \paren{
    \begin{array}{ccccc}
     2 & -2 & & & \\
    -1 & 2 & -1 & \phantom{\ddots} & \\
      & -1 & \ddots & \ddots & \\
      & & \ddots & 2 & -1 \\
      & & \phantom{\ddots} & -1 & 2
   \end{array}}
   \qquad \text{and} \qquad
   \dynkin[
    edge length=1.2cm,
    labels={n, 3, 2, \textcolor{red}{1}},
    backwards
  ]{B}{o.ooo},
\]
we visually represent this asymmetry
by coloring the first row of the array red.
The entire array has two overlapping
types of diamonds:
the modified condition
$ad - \textcolor{red}{b}^2 c = 1$
applies exclusively
to sub-diamonds whose top entry lies
in the (red) first row,
while the classical rule $ad - bc = 1$
is applied for all other adjacent sub-diamonds
further down the array.
Hence a positive integral frieze
$F = (x_{i, j})_{1 \leq i \leq n, j \in \Z}$
of type $B_n$ is represented by an array
\begin{equation}
  \label{eq:Bn-frieze}
  \begin{tikzcd}[row sep=.01in, column sep = .03in, font={\small}]
      \ldots & & 1 & & 1 & & 1 & & 1 & & 1 & & \ldots \\
      & \ldots & & \textcolor{red}{x_{1, 1}} & & \textcolor{red}{x_{1, 2}} & & \textcolor{red}{x_{1, 3}} & & \textcolor{red}{x_{1, 4}} & & \ldots \\
      \ldots & & x_{2, 0} & & x_{2, 1} & & x_{2, 2} & & x_{2, 3} & & x_{2, 4} & & \ldots \\
      & \ldots & & x_{3, 0} & & x_{3, 1} & & x_{3, 2} & & x_{3, 3} & & \ldots \\
      & & \ddots & & \ddots & & \ddots & & \ddots & & \ddots & & \ddots \\
      & \ldots & & x_{n, -1} & & x_{n, 0} & & x_{n, 1} & & x_{n, 2} & & \ldots \\
      \ldots & & 1 & & 1 & & 1 & & 1 & & 1 & & \ldots
  \end{tikzcd}
\end{equation}
whose sub-diamonds
\[
  \begin{tikzcd}[row sep=.01in, column sep = .03in, font={\small}]
    & b & \\
    a & & d \\
    & c & \\
  \end{tikzcd}
  \qquad \text{ and } \qquad
  \begin{tikzcd}[row sep=.01in, column sep = .03in, font={\small}]
    & \textcolor{red}{f} & \\
    e & & h \\
    & g & \\
  \end{tikzcd}
\]
satisfy the equalities $ad - bc = 1$ and $eh - \textcolor{red}{f}^2 g = 1$, respectively.

\begin{example}
  \label{ex:B6-frieze}
  The following is a frieze of type $B_6$
  with largest entry
  $F_{7} F_{8} - 1 = 272$.
  We insert dashes in the first diagonal ``column''
  to illustrate the appearance of
  the Dynkin diagram of type $B_6$.
  \begin{equation*}
    \begin{tikzcd}[row sep=.015in, column sep = .05in, font={\small}]
        \ldots & & 1 & & 1 & & 1 & & 1 & & 1 & & 1 & & 1 & & \ldots \\
        & \ldots & & \textcolor{red}{13} & & \textcolor{red}{21} & & \textcolor{red}{8} & & \textcolor{red}{3} & & \textcolor{red}{1} & & \textcolor{red}{2} & & \textcolor{red}{5} & &  \ldots \\
        \ldots & & 64 & & 272 \arrow[ul, Rightarrow] \arrow[dr, no head] & & 167 & & 23 & & 2 & & 1 & & 9 & & \ldots \\
        & \ldots & & 103 & & 103 \arrow[dr, no head] & & 60 & & 5 & & 1 & & 2 & & 23 & & \ldots \\
        \ldots & & 37 & & 39 & & 37 \arrow[dr, no head] & & 13 & & 2 & & 1 & & 5 & & \ldots \\
        & \ldots & & 14 & & 14 & & 8 \arrow[dr, no head] & & 5 & & 1 & & 2 & & 8 & & \ldots \\
        \ldots & & 3 & & 5 & & 3 & & 3 & & 2 & & 1 & & 3 & & \ldots \\
        & \ldots & & 1 & & 1 & & 1 & & 1 & & 1 & & 1 & & 1 & & \ldots 
    \end{tikzcd}
  \end{equation*}
  Notice that all sub-diamonds
  $\begin{smallmatrix}
    & \textcolor{red}{f} & \\
    e & & h \\
    & g &
  \end{smallmatrix}$
  involving the first row (red) at the top satisfy
  $eh - f^2g = 1$, while all other sub-diamonds
  $\begin{smallmatrix}
    & b & \\
    a & & d \\
    & c &
  \end{smallmatrix}$
  satisfy $ad - bc = 1$.
\end{example}

For type $D_n$ ($n \geq 4$), the Dynkin diagram
contains a single bifurcation at the trivalent node.
With the Cartan matrix and corresponding labeling
for the Dynkin diagram of type $D_n$
\[
  \setlength{\arraycolsep}{4pt} 
  \renewcommand{\arraystretch}{0.5}
  C = \paren{
    \begin{array}{ccccccc}
     2 & & -1 & \phantom{\ddots} & & \\
       & 2 & -1 & \phantom{\ddots} & & \\
    -1 & -1 & 2 & -1 & \phantom{\ddots} \\
      & & -1 & \ddots & \ddots & \\
      & & & \ddots & 2 & -1 \\
      & & & \phantom{\ddots} & -1 & 2 & \\
   \end{array}}
   \qquad \text{and}\qquad
   \dynkin[
    edge length=1.2cm,
    labels={n, \textcolor{violet}{4}, \textcolor{blue}{3}, \textcolor{red}{1}, \textcolor{orange}{2}},
    backwards
  ]{D}{o.oooo},
\]
the general $D_n$ frieze incorporates the
bifurcation by interlacing the first two rows.
These shapes overlap to tile the entire array:
the $5$-term relation is applied exactly where the
interlaced rows $1$, $2$, $3$, and $4$ meet,
whereas all other adjacent sub-diamonds in
rows $3$ through $n$ follow
the $ad - bc = 1$ rule.
Hence a positive integral frieze
$F = (x_{i, j})_{1 \leq i \leq n, j \in \Z}$
of type $D_n$ is represented by an array
\begin{equation}
  \label{eq:Dn-frieze}
  \begin{tikzcd}[row sep=.015in, column sep = .05in, font={\small}]
      \ldots & & 1 & & 1 & & 1 & & 1 & & 1 & & \ldots \\
      & \ldots & & \textcolor{orange}{x_{2, 1}} & & \textcolor{orange}{x_{2, 2}} & & \textcolor{orange}{x_{2, 3}} & & \textcolor{orange}{x_{2, 4}} & & \ldots & & \\
      \ldots & & \textcolor{blue}{x_{3, 0}} & \textcolor{red}{x_{1, 1}} & \textcolor{blue}{x_{3, 1}} & \textcolor{red}{x_{1, 2}} & \textcolor{blue}{x_{3, 2}} & \textcolor{red}{x_{1, 3}}& \textcolor{blue}{x_{3, 3}} & \textcolor{red}{x_{1, 4}}& \textcolor{blue}{x_{3, 4}} & & \ldots \\
      & \ldots & & \textcolor{violet}{x_{4, 0}} & & \textcolor{violet}{x_{4, 1}} & & \textcolor{violet}{x_{4, 2}} & & \textcolor{violet}{x_{4, 3}} & & \ldots & & \\
      \ldots & & x_{5, -1} & & x_{5, 0} & & x_{5, 1} & & x_{5, 2} & & x_{5, 3} & & \ldots \\
      & \ddots & & \ddots & & \ddots & & \ddots & & \ddots & & \ddots \\
      \ldots & & x_{n, -2} & & x_{n, -1} & & x_{n, 0} & & x_{n, 1} & & x_{n, 2} & & \ldots \\
      & \ldots & & 1 & & 1 & & 1 & & 1 & & \ldots
  \end{tikzcd}
\end{equation}
whose sub-diamonds
\[
  \begin{tikzcd}[row sep=.01in, column sep = .03in, font={\small}]
    & b & \\
    a & & d \\
    & c & \\
  \end{tikzcd}
  \qquad \text{ and } \qquad
  \begin{tikzcd}[row sep=.01in, column sep = .03in, font={\small}]
    & \textcolor{orange}{f} & \\
    \textcolor{blue}{e} & \textcolor{red}{g} & \textcolor{blue}{i} \\
    & \textcolor{violet}{h} & \\
  \end{tikzcd}
\]
satisfy the equalities $ad - bc = 1$ and $\textcolor{blue}{e}\textcolor{blue}{i} - \textcolor{orange}{f}\textcolor{red}{g}\textcolor{violet}{h} = 1$, respectively.

\begin{example}
  \label{ex:D5-frieze}
  The following is a frieze of type $D_5$
  with largest entry
  $F_{5} F_{6} - 1 = 39$.
  We insert dashes in the first diagonal ``column''
  to illustrate the appearance of
  the Dynkin diagram of type $D_5$.
  \begin{equation}
  \label{eq:D5-frieze}
    \begin{tikzcd}[
      row sep=.015in,
      column sep=.05in,
      font=\small
    ]
      \ldots & & 1 & & 1 & & 1 & & 1 & & 1 & & 1 & & 1 & & \ldots \\
      & \ldots & & \textcolor{orange}{8} \arrow[dr, no head] & & \textcolor{orange}{5} & & \textcolor{orange}{2} & & \textcolor{orange}{1} & & \textcolor{orange}{3} & & \textcolor{orange}{8} & & \ldots \\
      \ldots & & \textcolor{blue}{23} & \textcolor{red}{8} & \textcolor{blue}{39} \arrow[dr, no head] \arrow[l, no head] & \textcolor{red}{5} & \textcolor{blue}{9} & \textcolor{red}{2}& \textcolor{blue}{1} & \textcolor{red}{1}& \textcolor{blue}{2} & \textcolor{red}{3}& \textcolor{blue}{23} & \textcolor{red}{8}& \textcolor{blue}{39} &  & \ldots \\
      & \ldots & & \textcolor{violet}{14} & & \textcolor{violet}{14} \arrow[dr, no head] & & \textcolor{violet}{2} & & \textcolor{violet}{1} & & \textcolor{violet}{5} & & \textcolor{violet}{14} & & \ldots \\
      \ldots & & 3 & & 5 & & 3 & & 1 & & 2 & & 3 & & 5 & & \ldots \\
      & \ldots & & 1 & & 1 & & 1 & & 1 & & 1 & & 1 & & \ldots
    \end{tikzcd}
  \end{equation}
  Notice that all $5$-term sub-diamonds
  $\begin{smallmatrix}
      & \textcolor{orange}{f} & \\
    \textcolor{blue}{e} & \textcolor{red}{g} & \textcolor{blue}{i} \\
    & \textcolor{violet}{h} & \\
  \end{smallmatrix}$
  involving the first row (red)
  in the middle satisfy $ei - fgh = 1$,
  while all other $4$-term sub-diamonds
  $\begin{smallmatrix}
    & b & \\
    a & & d \\
    & c &
  \end{smallmatrix}$
  satisfy $ad - bc = 1$.
\end{example}

\begin{remark}
  An interactive tool for
  generating examples of frieze patterns
  of small rank
  is available at \cite{visual-ca}.
\end{remark}

\subsection{Frieze points}
\label{sec:frieze-points}
We recall the correspondence between Dynkin friezes
and points on affine cluster varieties,
largely following the presentation
in \cite{zhang-2025}.
It will be very useful to view friezes
as points on the following affine variety
(ultimately coming from the
lower bound model
in the theory of acyclic cluster algebras
\cite{fomin-zelevinsky-3}).
\begin{definition}
  \label{def:frieze-variety}
  For a Dynkin type $\Delta_n$
  with generalized Cartan matrix
  $C = (c_{i,j})_{1 \leq i, j \leq n}$,
  consider the $n$ polynomials
  \[
    f_{\Delta_n, i}
      := x_i y_{i}
        - \prod_{j = 1}^{i - 1} x_j^{-c_{j, i}}
        - \prod_{j = i + 1}^n x_j^{-c_{j, i}}
        \in \Z[x_1, \ldots, x_n, y_1, \ldots, y_n],
  \]
  and define the affine variety
  $Y_{\Delta_n} := V(f_{{\Delta_n}, 1}, \ldots, f_{{\Delta_n}, n})$
  to be their zero locus.
\end{definition}

If $x \in Y_{\Delta_n}$ then we call
$x$ a \textit{frieze point} of type ${\Delta_n}$.
When we say 
$x = (x_1, \ldots, x_n; y_1, \ldots, y_n)
\in Y_{\Delta_n}$ is a \textit{positive integral point}, we simply mean that
each coordinate $x_i$ and $y_i$
is a positive integer.
There is a bijection between $\frieze({\Delta_n}, \N)$
and the set $Y_{\Delta_n}(\N)$
of positive integral frieze points
due to de Saint Germain--Huang--Lu \cite{dhl}
(cf. \cite[Proposition 2.4]{zhang-2025});
we use this identification
between friezes and frieze points throughout.
\begin{proposition}[{\cite[Section 6.3]{dhl}}]
  \label{prop:correspondence}
  Let ${\Delta_n}$ be a finite Dynkin type.
  There is a one-to-one correspondence between
  the $\frieze({\Delta_n}, \N)$
  and $Y_{\Delta_n}(\N)$.
\end{proposition}

Concretely, a point $(x_1, \ldots, x_n; y_1, \ldots y_n) \in Y_{\Delta_n}(\N)$
corresponds to a positive integral frieze of type ${\Delta_n}$
with first diagonal $F_{i, 1} = x_i$ and
$n$ other entries given by the $y_i$.

\subsubsection*{Type $B_n$}
The affine variety $Y_{B_n}$
is the vanishing locus of the system
of equations:
\begin{align}
  x_1 y_1 &= x_2 + 1, \tag{B1}\label{eq:B1}\\
  x_2 y_2 &= x_1^2 + x_3, \tag{B2}\label{eq:B2}\\
  x_i y_i &= x_{i-1} + x_{i+1} \quad (3 \le i \le n-1), \tag{B3}\label{eq:B3}\\
  x_n y_n &= x_{n-1} + 1. \tag{B4}\label{eq:B4}
\end{align} 

\begin{example}
  \label{ex:B6-frieze-point}
  The positive integral frieze of type $B_6$
  in Example \ref{ex:B6-frieze}
  corresponds to the point
  \[
    x = (13, 272, 103, 37, 8, 3; 21, 1, 3, 3, 5, 3) \in Y_{B_6}(\N).
  \]
  Its horizontal translation by one diagonal
  (i.e. its image under the cluster Donaldson--Thomas transformation $\tau$)
  corresponds to the point
  \[
    \tau(x) = (21, 167, 60, 13, 5, 2; 8, 3, 3, 5, 3, 3) \in Y_{B_6}(\N).
  \]
\end{example}

\subsubsection*{Type $D_n$}
The affine variety $Y_{D_n}$
is the vanishing locus of the system
of equations:
\begin{align}
  x_1 y_1 &= x_3 + 1, \tag{D1}\label{eq:D1}\\
  x_2 y_2 &= x_3 + 1, \tag{D2}\label{eq:D2}\\
  x_3 y_3 &= x_1 x_2 + x_4, \tag{D3}\label{eq:D3}\\
  x_i y_i &= x_{i-1} + x_{i+1} \quad (4 \le i \le n-1), \tag{D4}\label{eq:Dmid}\\
  x_n y_n &= x_{n-1} + 1. \tag{D5}\label{eq:Dlast}
\end{align}

\begin{example}
  \label{ex:D5-frieze-point}
  The positive integral frieze of type $D_5$
  in Example \ref{ex:D5-frieze}
  corresponds to the point
  \[
    x = (8, 8, 39, 14, 3; 5, 5, 2, 3, 5) \in Y_{D_5}(\N).
  \]
  Its horizontal translation by one diagonal
  (i.e. its image under the cluster Donaldson--Thomas transformation $\tau$)
  corresponds to the point
  \[
    \tau(x) = (5, 5, 9, 2, 1; 2, 2, 3, 5, 3) \in Y_{D_5}(\N).
  \]
\end{example}

\section{Universality}
\label{sec:universality}

We give an elementary
construction of positive integral friezes
to prove Theorem \ref{thm:universality}
and show that all positive integers
appear in Dynkin friezes.
\begin{proof}[Proof of Theorem \ref{thm:universality}]
  For any integer $n \geq 2$,
  consider the following points in the affine varieties
  $Y_{{\Delta_n}}(\N)$, which correspond to friezes via Proposition \ref{prop:correspondence}:
  \begin{align*}
    (n + 1, n, n-1, \ldots, 2; 1, 2, 2, \ldots, 2)
      &\in Y_{A_n}(\N), \\
    (n + 1, n, n-1, \ldots, 3, 2; 1, n + 3, 2, \ldots, 2)
      &\in Y_{B_n}(\N), \\
    (n^2 + 1, n, n-1, \ldots, 3, 2; 1, n + 1, 2, \ldots, 2)
      &\in Y_{C_n}(\N), \\
    (n, n, n-1, n-2, \ldots, 3, 2; 1, 1, n + 2, 2, \ldots, 2)
      &\in Y_{D_n}(\N).
  \end{align*}
  The arbitrary integer $n$ is a coordinate in each of these points,
  and hence there is a frieze of type $A_n$,
  $B_n$, $C_n$, and $D_n$
  containing $n$ as an entry
  by Proposition \ref{prop:correspondence}.
\end{proof}

\begin{remark}
  An immediate consequence of
  Theorem \ref{thm:universality}
  is that all prime numbers
  appear as entries of friezes.
  We note that \textit{a priori},
  the Fibonacci prime conjecture
  would imply the infinitude of prime numbers
  that appear as entries of friezes
  of type $A_n$,
  which we will call \textit{frieze primes}
  of type $A_n$.
  A result of Cheah--de Saint Germain
  \cite[Theorem 1]{cheah-dsg}
  says that for every positive integer $n > 1$,
  there is a positive integral frieze of type $A_n$ (resp. $C_n$)
  with largest entry $F_{n+2}$ (resp. $F_{2n+1}$).
  Hence, the value set of positive integral friezes
  of type $A_n$
  contains all Fibonacci primes.
  While there are heuristics in analytic number theory
  on whether there are infinitely many Fibonacci primes
  (cf. \cite{bugeaud-luca-mignotte-siksek,grantham-granville}),
  the problem is still open (cf. \cite[Section 4.13]{cai}).
\end{remark}

%%%%%%%%%%%%%%%%%%%%%%%%%%%%%%%%%%%%%%%%%%%%%%%%%%%

\section{Maximal friezes}
\label{sec:maximal}

\subsection{Maximal \texorpdfstring{$B_n$}{Bn} friezes}
We now turn to the first part of
Theorem \ref{thm:max-bound},
which follows from
Proposition \ref{prop:correspondence}
and the following
explicit construction of
positive integral frieze points of type $B_n$
with maximal coordinate $F_{n+1}F_{n+2} - 1$.
Recall that the $n$-dimensional affine variety
$Y_{B_n} \subset \A^{2n}$
is defined by Equations \ref{eq:B1}--\ref{eq:B4}.

\begin{theorem}
  \label{thm:Bn-construction-fibonacci}
  Let $n \ge 3$.
  There exists a positive integral point
  \[
    x = (x_1, \ldots, x_n; y_1, \ldots, y_n) \in Y_{B_n}(\N)
  \]
  such that
  \[
    x_2 = F_{n+1} F_{n+2} - 1.
  \]
  Moreover, $x_2$ is the largest coordinate
  of the point $x$.
\end{theorem}

\subsubsection*{Proof of Theorem \ref{thm:Bn-construction-fibonacci}}

\,

\medskip\noindent
\textit{Small cases.}
First, we prove Theorem \ref{thm:Bn-construction-fibonacci}
for small $n$ explicitly.
By a direct computation on $Y_{B_n}(\N)$
(see Appendix \ref{app:maximal}),
the following tuples $(x_1, \ldots x_n; y_1, \ldots, y_n)$
are positive integral solutions of Equations \ref{eq:B1}--\ref{eq:B4}:

\[
  \begin{array}{ll}
    B_3: & (3,14,5;\,5,1,3),\\[0.5mm]
    B_4: & (5,39,14,3;\,8,1,3,5),\\[0.5mm]
    B_5: & (8,103,39,14,3;\,13,1,3,3,5),\\[0.5mm]
    B_6: & (13,272,103,37,8,3;\,21,1,3,3,5,3),\\[0.5mm]
    B_7: & (21,713,272,103,37,8,3;\,34,1,3,3,3,5,3),\\[0.5mm]
    B_8: & (34,1869,713,270,97,21,8,3;\,55,1,3,3,3,5,3,3),\\[0.5mm]
    B_9: & (55,4894,1869,713,270,97,21,8,3;\,89,1,3,3,3,3,5,3,3).
  \end{array}
\]

Furthermore, these tuples uniquely determine friezes
by Proposition \ref{prop:correspondence}.
In each case $x_2=F_{n+1}F_{n+2}-1$.
This proves Theorem \ref{thm:Bn-construction-fibonacci}
for $3 \leq n \leq 9$.
For the remainder of the proof we assume $n \geq 10$
and give a uniform construction and 
analysis for all larger $n$ using Fibonacci identities and $2\times 2$ matrices.

\medskip\noindent
\textit{An explicit frieze for large $n$.}
Set
\begin{equation}\label{eq:front}
  x_1 := F_{n+1},\quad y_1 := F_{n+2},\quad
  x_2 := F_{n+1}F_{n+2}-1,\quad y_2:=1.
\end{equation}
Then Equations \ref{eq:B1} and \ref{eq:B2} hold; furthermore
\begin{equation}\label{eq:x3}
  x_3 := x_2y_2 - x_1^2 = F_{n+1}F_n - 1.
\end{equation}

Define
\[
  t(n):=3+\Big\lfloor\frac{n-1}{2}\Big\rfloor,
\]
and prescribe
\begin{equation}\label{eq:y-pattern}
  y_i :=
  \begin{cases}
    F_{n+2}, & i=1,\\[0.5mm]
    1, & i=2,\\[0.5mm]
    3, & 3\le i\le n-1,\ i\neq t(n),\\[0.5mm]
    5, & i=t(n),
  \end{cases}
\end{equation}
leaving $y_n$ to be determined. Recursively define
\begin{equation}
  \label{eq:x-rec}
  x_{i+1}:=y_i x_i - x_{i-1}, \qquad 3 \leq i \leq n-1,
\end{equation}
so that Equation \ref{eq:B3} holds for $3 \leq i \leq n-1$.
This defines integers
$x_3, \ldots, x_n$
(\textit{a priori} rational numbers),
and finally set
\begin{equation}\label{eq:yn-def}
  y_n:=\frac{x_{n-1}+1}{x_n},
\end{equation}
so Equation \ref{eq:B4} holds identically.
To finish we must show $y_n \in \N$ and
$x_i > 0$ for all $i$.

\medskip\noindent
\textit{Integrality and positivity.}
We will encode various Fibonacci identities
using matrices.
For $i \geq 3$ set
\[
  A_i:=\begin{pmatrix}y_i&-1\\[0.5mm]1&0\end{pmatrix} \quad \text{ and } \quad
  w_i:=\begin{pmatrix}x_i\\ x_{i-1}\end{pmatrix}.
\]
Then Equation \ref{eq:x-rec} is $w_{i+1}=A_iw_i$, which yields the product
\begin{equation}
  \label{eq:wn-product}
  w_n
    = A_{n-1}\cdots A_3\, w_3,
  \qquad \text{ where }
  w_3
    =\begin{pmatrix}x_3\\ x_2\end{pmatrix}
    =\begin{pmatrix}F_{n+1}F_n-1\\[0.5mm]F_{n+1}F_{n+2}-1\end{pmatrix}.
\end{equation}

To evaluate this product, we split the transformation matrices into two types based on their trace. Let
\[
  B:=\begin{pmatrix}3&-1\\[0.5mm]1&0\end{pmatrix} \quad \text{ and } \quad
  C:=\begin{pmatrix}5&-1\\[0.5mm]1&0\end{pmatrix}.
\]
By the definition of $y_i$ in Equation \ref{eq:y-pattern}, the sequence of matrices $A_{n-1}\cdots A_3$ consists entirely of $B$ matrices, except for exactly one $C$ matrix located at index $t(n)$. Counting the number of $B$ matrices to the left and right of this $C$ gives an explicit factorization, which we formalize in the following lemma.

\begin{lemma}
  \label{lem:tail-computation}
  Let $n \geq 6$, let $w_3 = \begin{psmallmatrix}
            F_{n+1}F_n-1 \\
            F_{n+1}F_{n+2}-1
          \end{psmallmatrix} \in \Z^2$ as in
  Equation \ref{eq:wn-product},
  and define the pair of nonnegative integers
  \begin{align*}
    (\ell_1, \ell_2) :=
      \begin{cases}
        (m-1, m-3) & \text{ if } n = 2m, \\
        (m, m-3) & \text{ if } n = 2m + 1.
      \end{cases}
  \end{align*}
  Then the matrix product from Equation \ref{eq:wn-product} factors as
  \[
    A_{n-1}\cdots A_3 = B^{\ell_2} C B^{\ell_1},
  \]
  and furthermore
  \[
    B^{\ell_2} C B^{\ell_1} w_3 =
      \begin{pmatrix}
        3 \\
        8 
      \end{pmatrix}.
  \]
\end{lemma}

\begin{proof}
  We expand
  \[
    B^{\ell_1} =
      \begin{pmatrix}
        F_{2\ell_1 + 2} & -F_{2\ell_1} \\
        F_{2\ell_1} & -F_{2\ell_1 - 2}
      \end{pmatrix},
  \]
  and multiply by the explicit vector $w_3$. Let $v := B^{\ell_1} w_3$. Using the designated values for $\ell_1$, the first coordinate simplifies to
  $v_1 = F_{2m}$ after applying Catalan's identity
  $F_{2m}^2 - F_{2m-2}F_{2m+2} = 1$
  (see \cite[p. xi]{cai} for instance).
  Applying $C$ to $v=(F_{2m},v_2)$ yields
  $Cv=(5F_{2m}-v_2,\,F_{2m})$, and substituting the explicit form of
  $v_2$ (computed as above) reduces $5F_{2m}-v_2$ to $F_{2m-2}$ by the same
  family of Fibonacci identities. Finally, applying $B^{\ell_2}$ to this vector
  and simplifying with d'Ocagne's identity
  \[
    F_{r+1}F_s - F_r F_{s+1} = (-1)^r F_{s-r}
  \]
  produces the constant vector $(3,8)$.
\end{proof}

By Equation \ref{eq:wn-product} and Lemma \ref{lem:tail-computation}, we can now evaluate the tail vector $w_n$:
\[
  \begin{pmatrix}
    x_n \\
    x_{n-1}
  \end{pmatrix}
    = w_n
    = A_{n-1} \cdots A_3 \, w_3
    = B^{\ell_2} C B^{\ell_1} w_3
    = \begin{pmatrix}
        3 \\
        8
      \end{pmatrix}.
\]
In particular, $x_n=3$ and $x_{n-1}=8$. By Equation \ref{eq:yn-def}, it follows immediately that
\[
  y_n = \frac{x_{n-1}+1}{x_n} = \frac{8+1}{3} = 3 \in \N.
\]

All $x_i$ and $y_i$ are integers: the recurrence Equation \ref{eq:x-rec} uses
only integer arithmetic starting from the integer ``front'' data
in Equation \ref{eq:front},
while $y_n$ was just shown to be an integer.

Positivity follows by showing that the sequence of $x_i$'s is strictly decreasing from $i=2$ down to $n$.
Using backward induction from the tail,
the base case is $x_{n-1} = 8 > 3 = x_n$.
For the induction step, let $3 \le i \le n-1$ and assume $x_i > x_{i+1}$.
Since $y_i \ge 3$ by Equation \ref{eq:y-pattern}, the recurrence in Equation \ref{eq:x-rec} gives
\[
  x_{i-1} = y_i x_i - x_{i+1} \;>\; y_i x_i - x_i = (y_i-1)x_i \;\ge\; 2x_i \;>\; x_i.
\]
Iterating downward yields $x_2 > x_3 > \cdots > x_n = 3$, ensuring $x_i \ge 3$ for all $2 \le i \le n$.
Because $x_1 = F_{n+1} > 0$ and each $y_i \in \{F_{n+2}, 1, 3, 5\}$ by construction, every $x_i$ and $y_i$ is strictly positive.
This establishes a valid positive integral solution to Equations \ref{eq:B1}--\ref{eq:B4} with $x_2 = F_{n+1}F_{n+2}-1$, completing the proof of Theorem \ref{thm:Bn-construction-fibonacci}.

\qed

\subsection{Maximal \texorpdfstring{$D_n$}{Dn} friezes}
We now prove the second part of
Theorem \ref{thm:max-bound},
which follows from
Proposition \ref{prop:correspondence}
and the following
explicit construction of
positive integral frieze points of type $D_n$
with maximal coordinate $F_{n}F_{n+1} - 1$.
Recall that the $n$-dimensional affine variety
$Y_{D_n} \subset \A^{2n}$
is defined by Equations \ref{eq:D1}--\ref{eq:Dlast}.
In particular,
we show that positive integral frieze points
of type $B_{n - 1}$ give rise to
positive integral frieze points of type $D_n$
via the map
\[
  \begin{tikzcd}[row sep = tiny]
    Y_{B_{n - 1}}(\N) \arrow[r] & Y_{D_n}(\N) \\
    (x_1', \ldots, x_{n-1}'; y_1', \ldots, y_{n-1}') \arrow[r, mapsto] & (x_1', x_1', x_2', x_3', \ldots, x_{n-1}';
      y_1', y_1', y_2', \ldots, y_{n - 1}').
  \end{tikzcd}
\]

\begin{theorem}
  \label{thm:Dn-construction-fibonacci}
  Let $n \geq 4$.
  There exists a positive integral point
  \[
    x = (x_1, \ldots, x_n; y_1, \ldots, y_n) \in Y_{D_n}(\N)
  \]
  such that
  \[
    x_3 = F_n F_{n+1} - 1.
  \]
  Moreover, $x_3$ is the largest coordinate of the
  point $x$.
\end{theorem}

\subsubsection*{Proof of Theorem \ref{thm:Dn-construction-fibonacci}}

\,

Fix an integer $n \geq 4$.
By Theorem \ref{thm:Bn-construction-fibonacci},
there exists a positive integer solution
\[
  (x_1', \ldots, x_{n-1}'; y_1', \ldots, y_{n-1}') \in \N^{2n-2}
\]
to Equations \ref{eq:B1}--\ref{eq:B4} with
\[
  x_2' = F_n F_{n+1} - 1.
\]

We now build a solution to the $D_n$ system
in Equations \ref{eq:D1}--\ref{eq:Dlast}
from this frieze point of type $B_{n-1}$
by the following explicit recipe:
\begin{equation}\label{eq:D-from-B}
  \begin{aligned}
    x_1 &:= x_1', &\quad y_1 &:= y_1',\\
    x_2 &:= x_1', &\quad y_2 &:= y_1',\\
    x_3 &:= x_2', &\quad y_3 &:= y_2',\\
    x_i &:= x_{i-1}', &\quad y_i &:= y_{i-1}'
    && (4 \leq i \leq n).
  \end{aligned}
\end{equation}
We check that Equations \ref{eq:D-from-B}
indeed gives a solution of
Equations \ref{eq:D1}--\ref{eq:Dlast}.

\begin{itemize}
  \item Using Equation \ref{eq:B1},
    we have $x_1' y_1' = x_2' + 1$. Therefore
    \[
      x_1 y_1 = x_1' y_1' = x_2' + 1 = x_3 + 1,
    \]
    so Equation \ref{eq:D1} holds, and similarly
    \[
      x_2 y_2 = x_1' y_1' = x_2' + 1 = x_3 + 1,
    \]
    so Equation \ref{eq:D2} holds as well.
  \item By Equation \ref{eq:B2},
    we have $x_2' y_2' = x_1'^2 + x_3'$. Then
    \[
      x_3 y_3
        = x_2' y_2'
        = x_1'^2 + x_3'
        = x_1 x_2 + x_4,
    \]
    since $x_1 = x_2 = x_1'$ and $x_4=x_3'$ by Equation \ref{eq:D-from-B}.
    Thus Equation \ref{eq:D3} holds.
  \item For $4 \leq i \leq n - 1$,
    we have $3 \leq i - 1 \leq n - 2$,
    so Equation \ref{eq:B3}
    applies at index $i - 1$:
    \[
      x_{i-1}' y_{i-1}' = x_{i-2}' + x_i'.
    \]
    Using Equation \ref{eq:D-from-B}, this becomes
    \[
      x_i y_i = x_{i - 1} + x_{i + 1},
    \]
    so Equation \ref{eq:Dmid}
    holds for all $4 \leq i \leq n-1$.
  \item Finally, Equation \ref{eq:B4} says
    \[
      x_{n - 1}' y_{n - 1}' = x_{n - 2}' + 1.
    \]
    Using Equation \ref{eq:D-from-B},
    $x_n = x_{n - 1}'$, $y_n = y_{n-1}'$, and
    $x_{n - 1} = x_{n - 2}'$; hence
    \[
      x_n y_n = x_{n - 1}' y_{n - 1}'
        = x_{n - 2}' + 1
        = x_{n - 1} + 1,
    \]
    so Equation \ref{eq:Dlast} holds.
\end{itemize}

Therefore Equation \ref{eq:D-from-B}
transforms any positive solution of
the type $B_{n-1}$
system into a positive solution of
the type $D_n$ system.
Since the $B_{n - 1}$ solution
was positive integral,
all $x_i, y_i$ are clearly
positive integers.
Moreover,
\[
  x_3 = x_2' = F_n F_{n+1} - 1
\]
by construction.

In the particular $B_{n-1}$ solution
given by
Theorem \ref{thm:Bn-construction-fibonacci},
we know that
$x_2' = F_n F_{n+1}-1$
is the largest of the $x_i'$.
All other $x_i'$ are strictly smaller, and
all $y_i'$ are also bounded by smaller Fibonacci numbers.
Under the identification
\[
  (x_1, \ldots, x_n; y_1, \ldots, y_n)
    = (x_1', x_1', x_2', x_3', \ldots, x_{n-1}';
      y_1', y_1', y_2', \ldots, y_{n - 1}'),
\]
the maximum among the $x_i$ is still
$x_3 = x_2'$.
The $y_i$ are also bounded by $x_2'$.
Thus in this frieze point of type $D_n$,
$x_3$ is indeed the largest coordinate.
Hence we have proved Theorem
\ref{thm:Dn-construction-fibonacci}.

\qed

\section{Exponential growth and minimal ranks}
\label{sec:exponential-growth}

We briefly note that
maximal frieze entries
exhibit exponential growth
in the classical families.

By Binet's formula
(originally due to de Moivre,
see \cite[p. xii]{cai} for instance),
\[
  F_n = \frac{\varphi^n - \psi^n}{\sqrt{5}},
\]
where $\varphi = \frac{1 + \sqrt{5}}{2}$
is the golden ratio
and $\psi = \frac{1 - \sqrt{5}}{2}$.
Since $\verts{\psi} < 1$, it follows
that the Fibonacci sequence
exhibits exponential growth
\[
  F_n \widesim{} \frac{\varphi^n}{\sqrt{5}}
  \qquad (n \rightarrow \infty).
\]
Combined with \cite[Theorem 1]{cheah-dsg}
and Theorem \ref{thm:max-bound},
we obtain the exponential growth
of maximal friezes in each classical family:
\begin{align*}
  u_{A_n, \N} &= F_{n+2}
    \widesim{} \frac{\varphi^{n+2}}{\sqrt{5}}, \\
  u_{B_n, \N} &\geq F_{n+1}F_{n+2} - 1
    \widesim{} \frac{\varphi^{2n+3}}{5}, \\
  u_{C_n, \N} &= F_{2n+1}
    \widesim{} \frac{\varphi^{2n+1}}{\sqrt{5}}, \\
  u_{D_n, \N} &\geq F_{n}F_{n+1} - 1
    \widesim{} \frac{\varphi^{2n+1}}{5}.
\end{align*}

Since an integer $k$ can occur
in a frieze of type $\Delta_n$ only if
$k \leq u_{\Delta_n, \N}$,
the bounds on $u_{\Delta_n, \N}$
give lower bounds
on the minimal rank
\[
  n_{\min}(\Delta, k)
  = \min\set{n \geq 2 \Mid k \text{ appears in some }
    F \in \frieze(\Delta_n, \N)}.
\]
For example, for type $A$,
one has
\[
  n_{\min}(A, k)
    \geq \log_\varphi\paren{\sqrt{5} \, k} - 2 + o(1),
\]
and for type $C$,
\[
  n_{\min}(C,k)
    \geq \frac{1}{2}
      \log_\varphi\paren{\sqrt{5} \, k}
      - \frac{1}{2} + o(1).
\]
Using the upper bounds of
\begin{align*}
  u_{B_n, \N} &\leq 2^{\frac{(n+1)^2(n-2)}{2}}, \\
  u_{D_n, \N} &\leq 2^{\frac{n^3}{2}},
\end{align*}
from \cite[Proposition 4.1]{zhang-2025},
we likewise obtain
\[
  n_{\min}(B, k), n_{\min}(D, k)
    \gg (\log k)^{1/3}.
\]
Thus, it might be the case that
the rank needed to realize
a prescribed integer grows much more slowly
than the upper bound
$n_{\min}(\Delta, k) \leq k$
deduced from Theorem \ref{thm:universality}
(viz. Remark \ref{rem:efficiency}).

%%%%%%%%%%%%%%%%%%%%%%%%%%%%%%%%%%%%%%%%%%%%%%%%%%%%%%%%%%%%%

\appendix
\newpage
\section{Maximal friezes for small \texorpdfstring{$n$}{n}}
\label{app:maximal}

In the following table, we list
frieze points of maximal height.
By \cite[Section 6.3]{dhl},
the first $n$ coordinates correspond
to the first diagonal $(F_{1, 1}, \ldots, F_{n, 1})$
in the frieze as an array of positive integers
(e.g. as in arrangement \ref{eq:D5-frieze} for $D_5$).
The data was computed by a brute-force search in C++
for points on $Y_{\Delta_n}(\N)$
up to the known values for $\#Y_{\Delta_n}(\N)$
in \cite[Table 1]{zhang-2025}.
Note that all maximal friezes of type ${\Delta_n}$
can be obtained as a horizontal translation
(i.e. in the $\tau$-orbit)
of frieze points of maximal height.

\begin{figure}[H]
  \centering
  \begin{tabular}{c|l}
    \textbf{Type} ${\Delta_n}$ & \textbf{Positive integral frieze points of type ${\Delta_n}$ with maximal height} \\
    \hline
    $B_2$ & $\begin{matrix*}[l] (2, 1; 1, 5) \\ (3, 2; 1, 5) \\ (3, 5; 2, 2) \\ (2, 5; 3, 1) \end{matrix*}$ \\
    \hline
    $B_3$ & $\begin{matrix*}[l] (3, 14, 5; 5, 1, 3) \\ (5, 14, 3; 3, 2, 5) \end{matrix*}$ \\
    \hline
    $B_4$ & $\begin{matrix*}[l] (5, 39, 14, 3; 8, 1, 3, 5) \\ (8, 39, 14, 3; 5, 2, 3, 5) \end{matrix*}$ \\
    \hline
    $B_5$ & $\begin{matrix*}[l] (8, 103, 39, 14, 3; 13, 1, 3, 3, 5) \\ (13, 103, 37, 8, 3; 8, 2, 3, 5, 3) \end{matrix*}$ \\
    \hline
    $B_6$ & $\begin{matrix*}[l] (13, 272, 103, 37, 8, 3; 21, 1, 3, 3, 5, 3) \\ (21, 272, 103, 37, 8, 3; 13, 2, 3, 3, 5, 3) \end{matrix*}$ \\
    \hline
    $B_7$ & $\begin{matrix*}[l] (21, 713, 272, 103, 37, 8, 3; 34, 1, 3, 3, 3, 5, 3) \\ (34, 713, 270, 97, 21, 8, 3; 21, 2, 3, 3, 5, 3, 3) \end{matrix*}$ \\
    \hline
    $B_8$ & $\begin{matrix*}[l] (34, 1869, 713, 270, 97, 21, 8, 3; 55, 1, 3, 3, 3, 5, 3, 3) \\ (55, 1869, 713, 270, 97, 21, 8, 3; 34, 2, 3, 3, 3, 5, 3, 3) \end{matrix*}$ \\
    \hline
    $B_9$ & $\begin{matrix*}[l] (55, 4894, 1869, 713, 270, 97, 21, 8, 3; 89, 1, 3, 3, 3, 3, 5, 3, 3) \\ (89, 4894, 1867, 707, 254, 55, 21, 8, 3; 55, 2, 3, 3, 3, 5, 3, 3, 3)\end{matrix*}$ \\
    \hline
    $D_4$ & $\begin{matrix*}[l] (3, 3, 14, 5; 5, 5, 1, 3) \\ (5, 5, 14, 3; 3, 3, 2, 5)\end{matrix*}$\\
    \hline
    $D_5$ & $\begin{matrix*}[l] (5, 5, 39, 14, 3; 8, 8, 1, 3, 5) \\ (8, 8, 39, 14, 3; 5, 5, 2, 3, 5)\end{matrix*}$ \\
    \hline
    $D_6$ & $\begin{matrix*}[l](8, 8, 103, 39, 14, 3; 13, 13, 1, 3, 3, 5) \\ (13, 13, 103, 37, 8, 3; 8, 8, 2, 3, 5, 3)\end{matrix*}$ \\
    \hline
    $D_7$ & $\begin{matrix*}[l](13, 13, 272, 103, 37, 8, 3; 21, 21, 1, 3, 3, 5, 3) \\ (21, 21, 272, 103, 37, 8, 3; 13, 13, 2, 3, 3, 5, 3)\end{matrix*}$ \\
    \hline
    $D_8$ & $\begin{matrix*}[l](21, 21, 713, 272, 103, 37, 8, 3; 34, 34, 1, 3, 3, 3, 5, 3) \\(34, 34, 713, 270, 97, 21, 8, 3; 21, 21, 2, 3, 3, 5, 3, 3) \end{matrix*}$ \\
    \hline
    $D_9$ & $\begin{matrix*}[l](34, 34, 1869, 713, 270, 97, 21, 8, 3; 55, 55, 1, 3, 3, 3, 5, 3, 3) \\ (55, 55, 1869, 713, 270, 97, 21, 8, 3; 34, 34, 2, 3, 3, 3, 5, 3, 3)\end{matrix*}$ \\
    \hline
    $D_{10}$ & $\begin{matrix*}[l](55, 55, 4894, 1869, 713, 270, 97, 21, 8, 3; 89, 89, 1, 3, 3, 3, 3, 5, 3, 3) \\ (89, 89, 4894, 1867, 707, 254, 55, 21, 8, 3; 55, 55, 2, 3, 3, 3, 5, 3, 3, 3)\end{matrix*}$ \\
    \hline
  \end{tabular}
  \caption{Maximal frieze data for $B_n$ and $D_n$ with small $n$.}
  \label{table:maximal-frieze-data}
\end{figure}

%%%%%%%%%%%%%%%%%%%%%%%%%%%%%%%%%%%%%%%%%%%%%%%%%%%%%%%%%%%%%

\newpage
\bibliography{bibliography}{}
\bibliographystyle{alpha}

\end{document}